\documentclass[10pt]{amsart}
\usepackage{amsfonts,amssymb,yhmath}
 
\newtheorem{thm}{Theorem}[section] 
\newtheorem{lem}[thm]{Lemma}

\newcommand{\on}[2]{\genfrac{}{}{0pt}{}{#1}{#2}}

\theoremstyle{remark}

\def\qq{\mathbb{Q}}
\def\pp{\mathbb{P}}
\def\rr{\mathbb{R}}
\def\zz{\mathbb{Z}}

\numberwithin{equation}{section}

\begin{document}

\title{One half log discriminant and division polynomials}

\author{Robin de Jong}


\address{Mathematical Institute, University of Leiden, PO Box 9512,
2300 RA Leiden, The Netherlands}

\email{rdejong@math.leidenuniv.nl}

\begin{abstract} L. Szpiro and T. Tucker recently proved that under mild conditions, the valuation of the minimal discriminant of an elliptic curve with semistable reduction over a discrete valuation ring can be expressed in terms of intersections between $n$-torsion and $2$-torsion, where $n$ tends to infinity. The argument of Szpiro and Tucker is geometric in nature. We give a proof based on the arithmetic of division polynomials, and generalize the result to the case of hyperelliptic curves.
\end{abstract}


\maketitle
\thispagestyle{empty}

\section{Introduction}

Let $K$ be a field of characteristic $p \neq 2$ endowed with a non-trivial discrete valuation, and let $O$ be the ring of integers of $K$. Let $E$ be an elliptic curve over $K$ given by a minimal equation $y^2 = f(x)$ with $f(x) \in O[x]$ a monic cubic separable polynomial. Let $\pp^1_O$ be the projective line over $O$. Let $D$ be the Zariski closure in $\pp^1_O$ of the scheme of zeroes of $f$ on $\pp^1_K$, and for each positive integer $n$ with $p \nmid n$ let $H_n$ be the Zariski closure in $\pp^1_O$ of the pushforward under $x \colon E \to \pp^1_K$ of the $n$-torsion minus the $2$-torsion in $E$.

In \cite{st1} L. Szpiro and T. Tucker proved the following theorem.
\begin{thm} \label{stthm} Assume that $E$ has semistable reduction over $K$. Let $\Delta$ be the discriminant of $f$. Then the formula:
\[ \lim_{\on{n \to \infty}{p \nmid n}} \frac{1}{n^2} (D,H_n)_\nu = \frac{1}{2} \nu(\Delta) \]
holds, where $\nu \colon K^* \to \zz$ is the normalized valuation of $K$ and where $(,)_\nu$ is the geometric intersection pairing on the arithmetic surface $\pp^1_O$.
\end{thm}
The underlying reduced scheme of $H_n$ can be conveniently described by a \emph{division polynomial} $\psi_n \in O[x]$ (cf. \cite{silarit}, Exercise 3.7). The polynomial $\psi_n$ has degree $(n^2-1)/2$ if $n$ is odd, degree $(n^2-4)/2$ if $n$ is even, and has leading coefficient $n$. An alternative way of writing the conclusion of the theorem is therefore that:
\[ \frac{1}{n^2}  \log| \mathrm{Res}(f,\psi_n^2)|_\nu \longrightarrow \frac{1}{2} \log |\Delta|_\nu \]
as $n \to \infty$ with $p \nmid n$, where $|\cdot|_\nu \colon K^* \to \rr^+$ is any absolute value determined by $\nu$. The proof in \cite{st1} of Theorem \ref{stthm} uses the geometry of the special fiber of the minimal regular model of $E$ over $O$.

Our purpose in this note is to show that Theorem \ref{stthm} can alternatively be derived from a study of the arithmetic of the division polynomials $\psi_n$. In particular we will remove the assumption that $E$ should have semistable reduction over $K$, as well as the assumption that $K$ should be a discretely valued field. In fact, using the generalized division polynomials introduced by D. Cantor \cite{ca}, to be explained below, we can even generalize the result to the context of hyperelliptic curves.

Let $g$ be a positive integer, and let $k$ be a field of characteristic $p$ where $p=0$ or $p \geq 2g+1$. Let $|\cdot|$ be an absolute value on $k$. Let $(X,o)$ be an elliptic curve or a pointed hyperelliptic curve of genus $g \geq 2$ over $K$, given by an equation $y^2=f(x)$ with $f(x) \in k[x]$ monic, separable and of degree $2g+1$, putting $o$ at infinity.
\begin{thm} \label{main}  Let $\psi_n \in k[x]$ be the $n$-th Cantor division polynomial of $(X,o)$ and let $\alpha \in k$ be a root of $f$. Then:
\[ \frac{1}{n^2} \log |\psi_n^2(\alpha)| \longrightarrow \frac{1}{2} \log |f'(\alpha)| \]
as $n \to \infty$. Here, only integers $n$ are taken with $p \nmid (n-g+1)\cdots(n+g-1)$. In particular, under the same assumptions we have:
\[ \frac{1}{n^2} \log |\mathrm{Res}(f,\psi_n^2)| \longrightarrow
\frac{1}{2} \log |\Delta| \]
as $n \to \infty$ where $\Delta $ is the discriminant of $f$.
\end{thm}
The motivation in \cite{st1} to study limits of intersection numbers as in Theorem \ref{stthm} is that, when working over a number field $K$, these limits are natural local non-archimedean heights associated to the scheme $D$. As $D$ consists only of torsion points, the global height vanishes; this is used in \cite{st1} to show that the total archimedean contribution to the height is equal to $\frac{1}{2}\log|N_{K/\qq}(\Delta)|$ where $N_{K/\qq}(\Delta)$ is the norm of $\Delta$ in $\zz$. Theorem \ref{main} provides local heights at each of the archimedean places too, and allows one to verify a posteriori that the global height is zero, by the product formula.

We note that the condition that $p \nmid (n-g+1)\cdots(n+g-1)$ is rather natural from the theory of Weierstrass points in positive characteristic (see \cite{ne} for example, esp. Remark 2.8). The condition generalizes the condition $p \nmid n$ from the case of elliptic curves.

\section{Cantor's division polynomials}

Our main result is a statement about the asymptotic behavior of certain special values of division polynomials associated to hyperelliptic curves. We briefly recall from \cite{ca} the construction of these division polynomials and their main properties. Let again $g\geq 1$ be an integer. 

Let $R$ be the commutative ring $\zz[a_1,\ldots,a_{2g+1}][1/2]$, 
where $a_1,\ldots,a_{2g+1}$ are indeterminates. Let $F(x)$ be the polynomial
$x^{2g+1}+a_{1}x^{2g} + \cdots+a_{2g} x + a_{2g+1}$ in $R[x]$, and 
let $\Delta \in R$ be the discriminant of $F$. Let $y$ be a variable satisfying $y^2 =F(x)$, and write $E_1(z)$ for the polynomial $E_1(z)=(F(x-z)-y^2)/z$ in $R[x,z]$. Put: 
\[ S(z)=(-1)^{g+1} y \sqrt{1+zE_1(z)/y^2} \, , \] 
where 
$\sqrt{1+zE_1(z)/y^2}$ is the power series in $R[x,y^{-1}][[z]]$ 
obtained by binomial expansion on $1+zE_1(z)/y^2$. 

It is straightforward to verify that: 
\[ S(z)^2 = F(x-z) \, ,  \quad \textrm{and} \qquad 
 S(z) = \sum_{j=0}^\infty P_j(x)(2y)^{1-2j} z^j \]
for some $P_j(x) \in R[x]$ of degree $2jg$ and with leading coefficient in $\zz$. 

Let $n \geq g$ be an integer. Then Cantor's division polynomial $\psi_n$ (in genus $g$) is by definition the element of $R[x]$ given by:
\begin{equation} \label{defpsi} 
\psi_n = \left\{ \begin{array}{cl}
 \left| \begin{array}{cccc} 
 P_{g+1} &  P_{g+2} & \cdots & P_{(n+g)/2} \\ 
 P_{g+2} & \adots & \adots & \vdots \\
 \vdots & \adots & \adots & P_{n-2} \\
 P_{(n+g)/2} & \cdots & P_{n-2} & P_{n-1} \end{array} \right|  
& n \equiv g \bmod 2 \, , \\
 \left| \begin{array}{cccc} 
 P_{g+2} &  P_{g+3} & \cdots & P_{(n+g+1)/2} \\ 
 P_{g+3} & \adots & \adots & \vdots \\
 \vdots & \adots & \adots & P_{n-2} \\
 P_{(n+g+1)/2} & \cdots & P_{n-2} & P_{n-1} \end{array} \right|  
& n \equiv g+1 \bmod 2  \, . \end{array} \right. 
\end{equation}
For $n=g$ and $n=g+1$ we agree that $\psi_n$ is the unit element.
We have:
\[ \deg \psi_n = \left\{ \begin{array}{ll}
g(n^2-g^2)/2 & n \equiv g \bmod 2 \, , \\
g(n^2-(g+1)^2)/2 & n \equiv g+1 \bmod 2 \, . \end{array} \right. \]
Next, denote by $b(n)$ the leading coefficient of $\psi_n$ in $R$. Then $b(n)$ is an integer, and we have:
\[ p \nmid (n-g+1)\cdots(n+g-1) \Rightarrow p \nmid b(n) \]
for each prime integer $p$. Moreover, the $b(n)$ are the values at the integers $n \geq g$ of a certain numerical polynomial $b \in \qq[x]$ which can be written down explicitly.
  
The geometric meaning of the $\psi_n$ is as follows. Let $k$ be a field of characteristic $p$ where either
$p=0$ or $p \geq 2g+1$. Note that in particular $p \neq 2$. 
Let $f(x) \in k[x]$ be a monic and separable polynomial
of degree $2g+1$, and let $(X,o)$ be the elliptic or pointed hyperelliptic curve
of genus $g$ over $k$ given by the equation $y^2=f(x)$. The point $o$ is meant to be the unique point at infinity of $X$. 

Let $J = \mathrm{Pic}^0 \, X$ be the jacobian of $X$. It comes equipped with a natural symmetric theta divisor, representing the classes $[q_1+\cdots+q_{g-1}-(g-1)o]$ in $J$ where $q_1,\ldots,q_{g-1}$ are points running through $X$. Also we have a natural Abel-Jacobi embedding $\iota \colon X \to J$ given by sending $p \mapsto [p-o]$. Let $[n] \colon J \to J$ be the multiplication-by-$n$ map on $J$. For integers $n$ such that $n \geq g$ and $p \nmid  (n-g+1)\cdots(n+g-1)$ we then put: 
\[ X_n = \iota^* [n]^* \Theta \, . \] 
Then $X_n$ is an effective divisor of degree $gn^2$ on $X$. In fact, $X_n$ is the scheme of Weierstrass points of the line bundle $\mathcal{O}_X(o)^{\otimes n+g-1}$ on $X$; we refer to \cite{ne} for a further study of such schemes. Note that $X_n$ is a generalization of the scheme of $n$-torsion points on an elliptic curve. In analogy to what we did in that case in the Introduction, we subtract from each $X_n$ the part coming from the hyperelliptic ramification points. More precisely we put:
\[ X_n^* = \left\{ \begin{array}{ll}
 X_n - X_g & n \equiv g \bmod 2 \, , \\
 X_n - X_{g+1} & n \equiv g+1 \bmod 2 \, . 
\end{array} \right.
\]
We have: 
\[ X_g = \frac{g(g-1)}{2} D + g o \, , \quad X_{g+1} = \frac{g(g+1)}{2} D \, ,
\]
where $D$ denotes the reduced divisor of degree
$2g+2$ on $X$ consisting of the hyperelliptic ramification points of $X$. 
It can be shown (in fact we will see a proof below) that these $X_n^*$ 
are effective $k$-divisors on $X$ with support disjoint from the hyperelliptic ramification points. Note that:
\[ \deg X_n^* = \left\{ \begin{array}{ll}
g(n^2-g^2) & n \equiv g \bmod 2 \, , \\
g(n^2-(g+1)^2) & n \equiv g+1 \bmod 2 \, . \end{array} \right. \]
We have the following theorem.
\begin{thm} (D. Cantor \cite{ca}) Let $n \geq g$ be an integer and suppose that $p \nmid (n-g+1)\cdots(n+g-1)$. Specialize the polynomial $\psi_n$ from equation (\ref{defpsi}) to a polynomial in $k[x]$, by sending $a_1,\ldots,a_{2g+1}$ to the coefficients of $f$. Then $X_n^*$ is equal to the scheme of zeroes of $\psi_n$ on $X$.
\end{thm}
We note that if $(X,o)$ is an elliptic curve, the polynomials $\psi_n$ with $n \geq 1$ coincide with the usual division polynomials from elliptic function theory (cf. \cite{silarit}, Exercise~3.7).

\section{Proof of Theorem \ref{main}}

We just evaluate the
determinants at the right hand side of equation $(\ref{defpsi})$ at $\alpha$, where $\alpha$ is a root of $F=x^{2g+1}+a_{1}x^{2g} + \cdots+a_{2g} x + a_{2g+1}$ in an algebraic closure $\overline{Q(R)}$ of the
fraction field $Q(R)$ of $R$, and then specialize to $k$. Let $c_m = \frac{1}{2m+1} \binom{2m+1}{m}$ for $m \geq
0$ be the $m$-th Catalan number.  
\begin{lem} Let $j\geq1$ be an integer. Then the identity:
\[ P_j(\alpha) = (-1)^g \cdot  c_{j-1} \cdot F'(\alpha)^j \]
holds in $R[\alpha]$.
\end{lem}
\begin{proof} Recall the relations:
\[ S(z)= \sum_{j=0}^\infty P_j(x) (2y)^{1-2j}z^j \, , \quad
S(z)^2 = F(x-z) \, . \] 
We claim that:
\begin{equation} \label{byinduction}
\frac{1}{j!} \frac{ d^j S(z)}{d \, z^j} =
\frac{R_j(x,z)}{(2S(z))^{2j-1}} \end{equation}
for some $R_j(x,z) \in Q(R)[x,z]$ with $R_j(\alpha,0)= -c_{j-1}
\cdot F'(\alpha)^j $, for all $j \geq 1$. This gives the desired result since
$S(0)=(-1)^{g+1} y$ and hence $P_j(x)=(-1)^{g+1} R_j(x,0)$.

To prove the claim we argue by induction on $j$. We have:
\[ \frac{d S}{dz} = - \frac{F'(x-z)}{2S(z)} \] 
and this settles the
case $j=1$ with $R_1(x,z)=-F'(x-z)$. Now assume that (\ref{byinduction}) holds
with $R_j(x,z) \in Q(R)[x,z]$, and with $R_j(\alpha,0)=-c_{j-1} \cdot F'(\alpha)^j$
for some $j \geq 1$. Then a small calculation yields:
\[ \frac{1}{(j+1)!} \frac{d^{j+1} S}{dz^{j+1}} =
\frac{1}{j+1} \frac{d}{d z} \frac{R_j(x,z)}{(2S(z))^{2j-1}} =
\frac{R_{j+1}(x,z)}{(2S(z))^{2j+1}} \]
where:
\[ R_{j+1}(x,z) = \frac{2}{j+1} \left( 2 \left( \frac{d}{d z}
R_j(x,z) \right) F(x-z) + (2j-1)R_j(x,z) F'(x-z) \right) \, . \]
We find $R_{j+1}(x,z) \in Q(R)[x,z]$ and:
\begin{align*}
R_{j+1}(\alpha,0) & = \frac{2(2j-1)}{j+1} R_j(\alpha,0) \cdot F'(\alpha) \\
  & = -\frac{2(2j-1)}{j+1} c_{j-1} \cdot F'(\alpha)^{j+1} \\
  & = -c_j \cdot F'(\alpha)^{j+1}
\end{align*}
by the induction hypothesis. This completes the induction step.
\end{proof}
Evaluating equation $(\ref{defpsi})$ at $\alpha$ with the help of the Lemma yields the equality:
\begin{equation} \label{psi_atalpha} 
\psi_n(\alpha) = c(n) \cdot F'(\alpha)^{d(n)}   
\end{equation}
for all $n \geq g$ in $R[\alpha]$, where:
\begin{equation*}   c(n) = \left\{ \begin{array}{cl}
   \left| \begin{array}{cccc} 
 c_g &  c_{g+1} & \cdots & c_{(n+g)/2-1} \\ 
 c_{g+1} & \adots & \adots & \vdots \\
 \vdots & \adots & \adots & c_{n-3} \\
 c_{(n+g)/2-1} & \cdots & c_{n-3} & c_{n-2} \end{array} \right|  
& n \equiv g \bmod 2 \, , \\ \\
   \left| \begin{array}{cccc} 
 c_{g+1} &  c_{g+2} & \cdots & c_{(n+g-1)/2} \\ 
 c_{g+2} & \adots & \adots & \vdots \\
 \vdots & \adots & \adots & c_{n-3} \\
 c_{(n+g-1)/2} & \cdots & c_{n-3} & c_{n-2} \end{array} \right|  
& n \equiv g+1 \bmod 2 \, ,  \end{array} \right.  
\end{equation*}
at least up to a sign, and where $d(n) \in \zz_{\geq 0}$ is given by:
\[ d(n) = \left\{ \begin{array}{ll}
(n^2-g^2)/4  & n \equiv g \bmod 2 \, , \\
(n^2-(g+1)^2)/4   & n \equiv g+1 \bmod 2 \, . 
\end{array} \right.   \]
We claim that $p \nmid (n-g+1)\cdots (n+g-1) \Rightarrow
p \nmid c(n)$ holds for every prime number $p$ and every integer $n$ and
that the $c(n)$'s are the values at the integers $n \geq g$ of
a numerical polynomial $c \in \qq[x]$. The claim follows from a general result 
on Hankel determinants of Catalan numbers, due to M.~Desainte-Catherine and G.~Viennot (\cite{dcv}, Section~6). This result states that for arbitrary integers $l,m \geq 1$ we have the identity:
\[ 
\left| \begin{array}{cccc} 
 c_l &  c_{l+1} & \cdots & c_{l+m-1} \\ 
 c_{l+1} & \adots & \adots & \vdots \\
 \vdots & \adots & \adots & c_{l+2m-3} \\
 c_{l+m-1} & \cdots & c_{l+2m-3} & c_{l+2m-2} \end{array} \right|
= \prod_{1 \leq i \leq j \leq l-1} \frac{i+j +2m}{i+j} \, . 
\]
In particular $c(n)$ is non-vanishing in $k$ if the characteristic $p$ of $k$ satisfies $p \nmid (n-g+1)\cdots(n+g-1)$. And $c(n)$ has only polynomial growth in $n$.

Let us now place ourselves in the situation of Theorem \ref{main}. In particular we work over a field $k$ of characteristic $p$ with $p=0$ or $p \geq 2g+1$, and now $\alpha$ is a given root of $f \in k[x]$ in $k$. Let $n\geq g$ be an integer such that 
$p \nmid (n-g+1)\cdots(n+g-1)$.
From equation (\ref{psi_atalpha}) we obtain by specializing:
\begin{equation} \label{conclusion}  
\psi_n(\alpha) = c(n) \cdot f'(\alpha)^{d(n)} 
\end{equation}
in $k$. Since $f'(\alpha)$ and $c(n)$ are both non-zero in $k$ we deduce that
$\psi_n(\alpha)$ is non-zero in $k$ as well. In particular we find that $X_n^*$ has support disjoint from
the hyperelliptic ramification points, which proves a claim that we made earlier.
Theorem \ref{main} follows from equation (\ref{conclusion}) upon taking absolute values and logarithms, and letting $n$ tend to infinity, under the condition that $p \nmid (n-g+1)\cdots(n+g-1)$.  

\subsection*{Acknowledgments.}
The research done for this paper was supported by a VENI grant 
from the Netherlands Organisation for Scientific Research (NWO). Part of the research was done at the Max Planck Institute in Bonn, whose hospitality is greatly acknowledged.

\end{document}